\documentclass[a4paper, 12pt]{amsart}

\usepackage{amsmath}
\usepackage{amsfonts}
\newtheorem{thm}{Theorem}[section]

\newtheorem{lemma}[thm]{Lemma}

\newtheorem{defi}[thm]{Definition}

\newcommand{\R}{\mathbb{R}}
\newcommand{\Z}{\mathbb{Z}}
\providecommand{\norm}[1]{\lVert#1\rVert}
\title[Randomly perturbed self-affine sets]{Self-affine sets with non-compactly supported random perturbations}
\author{Thomas Jordan}
\address{Thomas Jordan\\School of Mathematics\\ The University of Bristol\\
University Walk\\Clifton\\ Bristol\\BS8 1TW\\UK}
\email{thomas.jordan@bristol.ac.uk}
\author{Natalia Jurga}
\address{Natalia Jurga\\School of Mathematics\\ The University of Bristol\\
University Walk\\Clifton\\ Bristol\\BS8 1TW\\UK}
\email{nj0529@bristol.ac.uk}
\thanks{The research carried out in this paper was supported by an undergraduate research bursary funded by the Nuffield Foundation and the London Mathematical Society. We would like to thank both organisations for their support. We also wish to thank Andrew Ferguson, K\'{a}roly Simon and Tuomas Sahlsten for useful conversations about this work. Finally we would like to thank the referee for providing detailed and helpful comments.}
\begin{document}
\begin{abstract}
In this note we consider the Hausdorff dimension of self-affine sets with random perturbations. We extend previous work in this area by allowing the random perturbation to be distributed according to distributions with unbounded support as long as the measure of the tails of the distribution decay super polynomially.  
\end{abstract}
\maketitle
\section{Introduction}
Calculating the Hausdorff dimension of self-affine sets has been an active area of research since the work of Bedford and McMullen \cite{bed,M}. While a lot of progress has been made in this time there are still several unresolved questions. Papers on Hausdorff dimension of self-affine sets tend to come in one of two types; either they calculate the Hausdorff dimension in some particular family as in \cite{bed} and \cite{M} or they obtain a result which holds for `generic' self-affine sets, \cite{Falcpaper}. The second approach started with a paper by Falconer, \cite{Falcpaper}, in which the Hausdorff dimension of self-affine sets was computed for almost all translations assuming the contraction rates are sufficiently small.

It was shown in \cite{tj} that if suitably defined random perturbations are added to a fixed self-affine set then an analogous result to Falconer's is obtained with no non-trivial assumption on the size of the contraction. However in \cite{tj} the random perturbations were assumed to be compactly supported so for example perturbations from a normal distribution could not be considered. In this note we show that this condition can be relaxed and be replaced by an assumption that the densities decay super polynomially, a natural  assumption which is satisfied by the normal distribution.

We consider a family $\mathcal{F}$ of affine contractions or iterated function system (IFS)
\begin{eqnarray}
\mathcal{F} := \{ f_i (x) = T_i \cdot x + a_i : i=1,\ldots, m\}
\label{ifs}
\end{eqnarray}
where $T_i \in \text{GL}_d(\mathbb{R})$ are such that $\norm{T_i} < 1$ for $1 \leq i \leq m$ and $a_i$ are vectors in $\mathbb{R}^d$. The following definition is standard:
\begin{defi}
Let $B$ be any large enough ball in $\mathbb{R}^d$ such that $f_i(B) \subseteq B$ for $1 \leq i \leq m$. Then the attractor of $\mathcal{F}$ is defined by the unique non-empty compact set for which 
\begin{displaymath}
\Lambda = \bigcup_{i=1}^m f_i(\Lambda)
\end{displaymath}
or alternatively,
\begin{displaymath}
\Lambda := \bigcap_{n=1}^{\infty} \bigcup_{i_0,\ldots,i_{n-1}} f_{i_0,\ldots ,i_{n-1}}(B)
\end{displaymath}
where $f_{i_0,\ldots ,i_{n-1}}= f_{i_0} \circ\cdots\circ f_{i_{n-1}}$ for $(i_0,\ldots,i_{n-1}) \in \{1,\ldots,m\}^n$.
\label{attractor}
\end{defi}

For an iterated function system (\ref{ifs}) we denote $$\norm{T}= \max \{\norm{T_i}: i=1,..., m\}$$ and $$\norm{a}= \max\{|a_i|: i=1,...,m\}.$$ 

We now introduce some notation and definitions in order to define the affinity dimension $d(T_1,\ldots,T_m)$ (In \cite{tj} this is referred to as the singularity dimension). When studying the family of contractions of the form (\ref{ifs}) we denote $J_{\infty}$ to be the set of all infinite words where $i_j \in \{1, \ldots, m\}$, $J_n$ to be the set of all finite words of length $n$ with $i_j \in \{1, \ldots, m\}$, $J$ to be the set of all finite words with $i_j \in \{1, \ldots, m\}$. For $\mathbf{i},\mathbf{j}\in J_{\infty}$ we denote $\mathbf{i} \wedge \mathbf{j}$ to be the truncation of $\mathbf{i}$ to the initial string that $\mathbf{i}$ and $\mathbf{j}$ agree on. For a string $\mathbf{i}= (i_0,...,i_{n}) \in J$ we 
denote 
$$[i_0,...,i_{n}]:=\{\textbf{j}\in J_{\infty}:j_m=i_m\text{ for all }0\leq m\leq n\}.$$
  Finally, for $\textbf{i}= (i_0, i_1,\ldots, i_n) \in J$ we denote $T_{\textbf{i}} = T_{i_0,\ldots, i_n} = T_{i_0}\cdots T_{i_n}$. 

Let $T: \mathbb{R}^d \to \mathbb{R}^d$ be a contracting, invertible, linear map, that is an invertible $d \times d$ matrix with matrix norm strictly less than $1$. We say that $\alpha$ is a singular value of $T$ if $\alpha$ is the positive square root of one of the eigenvalues of $T^{*}T$ , where $T^{*}$ denotes the transpose of $T$, or equivalently $\alpha$ is the length of one of the principal semi-axes of $T(B)$ where $B$ is the unit ball centred at the origin. 
We adopt the convention of denoting the $d$ singular values as $0< \alpha_d \leq \cdots\leq \alpha_1 < 1$. Sometimes, where it is not clear which matrix a singular value relates to, we will denote
$$0 < \alpha_d(T) \leq ... \leq \alpha_1(T) < 1$$
as the increasing singular values of  the $d \times d$ matrix $T$.

Fix $0\leq s\leq d$ and choose $r\in\mathbb{Z}$ such that $r-1<s\leq r$. We then define the singular value function as
\begin{displaymath}
\phi^s(T)= \alpha_1\alpha_2\cdots\alpha_{r-1}\alpha_r^{s-r+1}. 
\end{displaymath}
For $s>d$ we define $\phi^s(T)=(\alpha_1(T)\cdots\alpha_d(T))^{s/d}$. We are now ready to define the affinity dimension as follows,
\begin{defi}
For a self-affine IFS of the form (\ref{ifs}) we define the affinity dimension to be
$$d(T_1,\ldots,T_m) := \inf\{s > 0 : \sum_{n=0}^{\infty} \sum_{J_n} \phi^s(T_{i_0}\ldots T_{i_{n-1}}) < \infty\}$$
or equivalently, the value of s such that
$$ \lim_{n \to \infty}(\sum_{J_n} \phi^s(T_{i_0}\ldots T_{i_{n-1}}))^{\frac{1}{n}}= 1.$$
\label{s}
\end{defi}
In \cite{Falcpaper} Falconer gave the following almost sure result for the Hausdorff dimension of the attractor $\Lambda$ given that $\norm{T_i} < \frac{1}{3}$, whilst in \cite{s} Solomyak improved it to its present form with the weaker assumption that $\norm{T_i} < \frac{1}{2}$. Here almost sure is in terms of $md$-dimensional Lebesgue measure $\lambda_{md}$ on the translation vectors $(a_1,\ldots,a_m)\in\R^{dm}$.
\begin{thm}[Falconer]
Let
\begin{displaymath}
\mathcal{F} := \{ f_i (x) = T_i \cdot x + a_i : i=1,\ldots, m\}
\end{displaymath}
be an affine iterated function system with attractor $\Lambda$.
If 
$$\norm{T_i} < \frac{1}{2} \hspace{2mm} \forall 1 \leq i \leq m$$
then for $\lambda_{m \dot d}$-almost all vectors $\mathbf{a}:= (a_1,\ldots,a_m) \in \mathbb{R}^{m \dot d}$, the Hausdorff dimension of the attractor in Definition \ref{attractor} is
$$\mathrm{dim_H} \Lambda = d(T_1,\ldots,T_m)$$ where $d(T_1,\ldots,T_m)$ is as in Definition \ref{s}.
\end{thm}

In \cite{tj} it was shown that if a small random translation was allowed at each stage of the application of the contractions then the Hausdorff dimension of the resulting perturbed attractor will almost surely be equal to the affinity dimension. In particular, no restrictions on the norms of the maps was necessary. More precisely, it is assumed that at each application of the maps from the IFS we make a random additive error where these errors have distribution $\kappa$ where $\kappa$ is an absolutely continuous distribution with bounded density supported on a bounded disk $D$ which is centred at the origin. In particular, for $\mathbf{i_n}= (i_0,\ldots,i_{n-1}) \in J_n$ we denote 
$$f_{\mathbf{i_n}}^{\mathbf{x_{i_n}}} := (f_{i_0} + x_{i_0}) \circ (f_{i_1} + x_{i_0, i_1} ) \circ \ldots \circ (f_{i_{n-1}} + x_{i_0,\ldots,i_{n-1}})$$
where the elements of
$$\mathbf{x_{i_n}}:= (x_{i_0},\ldots, x_{i_0,\ldots,i_{n-1}}) \in D \times \ldots \times D$$
are independently and identically distributed with distribution $\kappa$. Let $\varphi(k)$ be the $k$-th element of the countable sequence
$$\{1, 2, \ldots, m, (1, 1), (1, 2), \ldots, (m, m), (1, 1, 1), \ldots \}$$
so that we can label the perturbations by the natural numbers $x_k= x_{\mathbf{i_n}}$ if $\varphi(k)=\mathbf{i_n}$. 
Then we can write the sequence of all random errors that perturb the attractor $\Lambda$ as $\mathbf{x}= \{x_k\}_{k \in \mathbb{N}} \in D^{\infty}$. This suggests the following definition
\begin{defi}
$$\Lambda_{\mathbf{x}} := \bigcap_{n=1}^{\infty} \bigcup_{\mathbf{i_n}} f_{\mathbf{i_n}}^{\mathbf{x_{i_n}}}(B)$$
where B is a ball, centred at the origin, which is sufficiently large such that $f_{\mathbf{i_n}}^{\mathbf{x_{i_n}}}(B)\subset B$ for all $\mathbf{x}\in D^{\infty}$ and $\mathbf{i}\in J_{\infty}$.
\label{attractor2}
\end{defi}
Letting $\mu$ denote the infinite product measure $\mu= \kappa \times \kappa \times \ldots$ on $D^{\infty}$ we can state the main result in \cite{tj}
\begin{thm}[Jordan-Pollicott-Simon]\label{JPS}
For a self-affine IFS of the form (\ref{ifs}), and for $\mu$-almost all $\mathbf{x} \in D^{\infty}$ then:
\begin{enumerate}
\item
 If $d(T_1,\ldots, T_m) \leq d$ then $\mathrm{dim_H \Lambda_{\mathbf{x}}} = d(T_1,\ldots, T_m)$
\item
 If $d(T_1,\ldots, T_m) > d$ then $\lambda_d(\Lambda_{\mathbf{x}}) > 0$, where $\lambda_d$ denotes $d$-dimensional Lebesgue measure.
\end{enumerate}
\end{thm}
In this note we build on the result of \cite{tj}, in that perturbations are no longer assumed to be taken from a bounded disk $D$. 
\begin{defi}\label{condition}
Let $\eta$ be an absolutely continuous distribution with bounded density supported on the space $Y=\mathbb{R}^d$ which satisfies the following condition: for every $k\in\mathbb{N}$ there exists a constant $c_k$ such that for all $t>0$, $\eta\{|X|> t\} \leq c_kt^{-k}$. ($\eta$ decays super-polynomially).
\end{defi}
 Such a measure $\eta$ will also have the following property: There exists a constant $K>0$ such that for any  $1 \leq n \leq d-1$ and set of orthonormal basis $\{z_1,\ldots, z_n\}$ if we let $\pi_n$ be a projection to the $n$-dimensional space spanned by $\{z_1,\ldots, z_n\}$ then the push-forward measure ${\pi_n}_{\ast} \eta$ must have some density $f$ with respect to $n$-dimensional Lebesgue measure, where $f\in L^{\infty}$ and $\norm{f}_{\infty}\leq K$.

Let
$$f_{\mathbf{i_n}}^{\mathbf{y_{i_n}}} := (f_{i_0} + y_{i_0}) \circ (f_{i_1} + y_{i_0, i_1} ) \circ \ldots \circ (f_{i_{n-1}} + y_{i_0,\ldots,i_{n-1}})$$
where the elements of
$$\mathbf{y_{i_n}}:= (y_{i_0},\ldots, y_{i_0,\ldots,i_{n-1}}) \in Y \times \ldots \times Y$$
are independently and identically distributed with distribution $\eta$. Define a random perturbation of the attractor $\Lambda$ as $\mathbf{y}= \{y_k\}_{k \in \mathbb{N}} \in Y^{\infty}$, numbered naturally as before, and put $\mathbb{P}$ as the infinite product measure $\mathbb{P}= \eta \times \eta \times \ldots$ on $Y^{\infty}$. Clearly, we can no longer define the attractor of this system in the same way as in Definition \ref{attractor2} since there does not exist a large enough ball $B$, as the distribution $\eta$ is not supported on a bounded disk. For the self-affine IFS of the form (\ref{ifs}) and the infinite word $\mathbf{i}= (i_0, i_1, \ldots) \in J_{\infty}$ set
\begin{eqnarray}
\Pi^{\mathbf{y}}(\mathbf{i}) &:=& \lim_{r \to \infty}(f_{i_0} +y_{i_0})\circ(f_{i_1} + y_{i_0, i_1})\circ\cdots\circ(f_{i_r} +y_{i_0,\ldots,i_r})(0) \label{pi1} \\
&=& a_{i_0} + y_{i_0} + \sum_{r=1}^{\infty}T_{i_0,\ldots ,i_{r-1}}(a_{i_r} + y_{i_0,\ldots,i_r}) \label{pi2}
\end{eqnarray}
if this limit exists.
In Lemma \ref{welldefined} we prove that for $\mathbb{P}$-almost all $\mathbf{y} \in Y^{\infty}$ this series converges for all $\mathbf{i}\in J_{\infty}$ thus we can define the attractor for $\mathbb{P}$-almost all $\mathbf{y} \in Y^{\infty}$ by
\begin{defi}
For a self-affine IFS of the form (\ref{ifs}), if for $\mathbf{y} \in Y^{\infty}$ we have that $\Pi^{\mathbf{y}}(i)$ is well defined for all $\mathbf{i}\in J_{\infty}$ then the associated attractor is defined to be
$$\Lambda^{\mathbf{y}} := \{ \Pi^{\mathbf{y}}(\mathbf{i}) :\mathbf{i}\in J_{\infty} \}.$$
\label{attractor3}
\end{defi}
 We can now state our main theorem
\begin{thm}\label{main}
For a self-affine IFS of the form (\ref{ifs}) and the attractor $\Lambda^{\mathbf{y}}$ defined in \ref{attractor3} we have that 
\begin{enumerate}
\item
 If $d(T_1, \ldots, T_m) \leq d$ then $\mathbb{P}$-almost surely, $\mathrm{dim_H}(\Lambda^{\mathbf{y}}) = d(T_1, \ldots, T_m)$
\item
 If $d(T_1, \ldots, T_m) > d$ then $\mathbb{P}$-almost surely $\lambda_d(\Lambda^{\mathbf{y}}) >0$.
\end{enumerate}
\end{thm}
In contrast to Theorem 2 in \cite{tj} this theorem holds when the random perturbations $y_{i_0,\ldots,i_{n-1}}$ are distributed according to a multivariate normal distribution. There are a couple of papers with similar results for random self-similar sets and measures, \cite{k} and \cite{pss} which look at a similar model of randomness for self-similar sets and measures. In \cite{k} random self-similar sets are considered where the translations are fixed but both the contraction rate and the amount of rotation varies randomly. In \cite{pss} measures which contract on average are considered, in this case the translation is fixed and the rate of contraction or expansion varies randomly according to a non-compactly supported distribution. In this case a different form of randomness is considered where there is less independence, the contraction or expansion rates at the $n$th level, $\tau_{i_0,\ldots,i_{n-1}}$, only depend on $i_{n-1}$.  There are also several papers on self-affine sets where the maps which form the iterated function system are chosen randomly at each level, usually based on a discrete distribution. This includes \cite{GL}, \cite{FM} and \cite{JJKKSS}. The work in \cite{GL} can be consider to be random analogues of Bedford-McMullen sets whereas in \cite{FM} and \cite{JJKKSS} the results are random versions of the almost everywhere results in \cite{Falcpaper}.  

The following sections are organised as follows: In Section 2 we show that $\mathbb{P}$-almost surely the attractor $\Lambda^{\mathbf{y}}$ in Definition \ref{attractor3} is well defined. The rest of that section is dedicated to getting a $\mathbb{P}$-almost sure upper bound for the Hausdorff dimension of the attractor. In Section 3 we use potential theoretic methods which follow from ideas of Falconer \cite{Falcpaper} in order to calculate a $\mathbb{P}$-almost sure lower bound.
In Section 4 we comment on some other possible applications of the methods in this note. 

\section{Proof of the upper bound for Theorem \ref{main}}
To prove the upper bound for Theorem \ref{main} we first need to show that $\Lambda^{\mathbf{y}}$ is well defined for $\mathbb{P}$ almost all $\mathbf{y}$. We then need to provide a suitable cover for such $\Lambda^{\mathbf{y}}$. We start with the following lemma which is important for both parts.
\begin{lemma}
For all $0< \theta < 1$, there exists a subset $X \subset Y^{\infty}$ such that $X$ has full measure and for all $\mathbf{y} \in X$ and $n$ sufficiently large we have
$$|y_{i_0,...,i_{n-1}}| \leq \frac{1}{\theta^n}$$
for all $(i_0,...,i_{n-1}) \in J_n$.
\label{bc}
\end{lemma}

\begin{proof}
We fix $\mathbf{y} \in Y^{\infty}$ and fix $0 < \theta < 1$. Put the event $A_n$ to be
$$A_n := \left\{|y_{i_0,\ldots,i_{n-1}}|> \frac{1}{\theta^n} \textnormal{ for some $(i_0,\ldots,i_{n-1}) \in J_n$}\right\}$$
and observe that
$$\mathbb{P}(A_n)\leq m^n\eta\{y\in\R^d:|y|> \theta^{-n}\}.$$
Fix $k$ such that $\theta^k < m^{-1}$. Then since $\eta$ satisfies Definition \ref{condition}, 
$$m^n\eta\left\{|X|> \frac{1}{\theta^n}\right\} \leq c_k \left(m\theta^k\right)^n.$$
Thus
\begin{displaymath}
\sum_{n=1}^{\infty}\mathbb{P} (A_n) < c_k \sum_{n=1}^{\infty} \left(m\theta^k\right)^n < \infty
\end{displaymath}
since $m\theta^k < 1$. 
Thus by the Borel-Cantelli lemma, for $\mathbb{P}$-almost every $\mathbf{y} \in Y^{\infty}$ the events $A_n$ occurs only finitely often, that is, there exists an $N \in \mathbb{N}$ such that for $n \geq N$, 
$$|y_{i_0,...,i_{n-1}}| \leq \frac{1}{\theta^n}.$$
\end{proof}

We first apply Lemma \ref{bc} to deduce that for $\mathbb{P}$-almost all $\mathbf{y} \in Y^{\infty}$ the attractor $\Lambda^{\mathbf{y}}$ in Definition \ref{attractor3} is well defined. 

\begin{lemma}
For $\mathbb{P}$-almost all $\mathbf{y} \in Y^{\infty}$ and all $\mathbf{i}\in J_{\infty}$ the series (\ref{pi2}) converges, thus the attractor in Definition \ref{attractor3} is well defined.
\label{welldefined}
\end{lemma}
\begin{proof}
 We want to show that 
$$\Pi^{\mathbf{y}}(\mathbf{i}) = a_{i_0} + y_{i_0} + \sum_{r=1}^{\infty}T_{i_0,\ldots, i_{r-1}}(a_{i_r} + y_{i_0,\ldots,i_r})$$
is convergent for $\mathbb{P}$-almost all $\mathbf{y} \in Y^{\infty}$, and all $\mathbf{i}\in J_{\infty}$. It is sufficient to show that
$$|\Pi|^{\mathbf{y}}(\mathbf{i}) := |a_{i_0} + y_{i_0}| + \sum_{r=1}^{\infty}|T_{i_0,\ldots ,i_{r-1}}(a_{i_r} + y_{i_0,\ldots,i_r})|$$
converges. Fix $\norm{T} < \theta < 1$. Then by Lemma \ref{bc} there exists $N \in \mathbb{N}$ such that for $n \geq N$, $(i_0,...,i_{n-1}) \in J_n$ and for $\mathbb{P}$-almost all $\mathbf{y} \in Y^{\infty}$ then $|y_{i_0,...,i_{n-1}}| < \frac{1}{\theta^n}$, in particular
$$|T_{i_0,\ldots ,i_{n-1}}(a_{i_n} + y_{i_0,\ldots,i_n})| \leq \norm{T}^n\left(\norm{a} + \frac{1}{\theta^{n+1}}\right).$$
Thus,
\begin{eqnarray*}
|\Pi|^{\mathbf{y}}(\mathbf{i}) &\leq& \sum_{n=1}^{N-2} \norm{T}^n (\norm{a} + |y_{i_0,\ldots,i_n}|) + \sum_{n=1}^{\infty} \norm{T}^n \left(\norm{a} + \frac{1}{\theta^{n+1}}\right)\\
&\leq& \sum_{n=1}^{N-2} \norm{T}^n (\norm{a} + |y_{i_0,\ldots,i_n}|) + \sum_{n=1}^{\infty} \norm{T}^n \norm{a} + \frac{1}{\theta}\sum_{n=1}^{\infty} \left(\frac{\norm{T}}{\theta}\right)^n
\end{eqnarray*}
which converges since $\frac{\norm{T}}{\theta} < 1$. Thus the series (\ref{pi2}) is absolutely convergent, thus convergent and the attractor is well defined for almost every $\mathbf{y} \in Y^{\infty}$. 
\end{proof}
Throughout the remainder of this section we work with a self-affine IFS of the form (\ref{ifs}) and assume that the perturbation $\mathbf{y}$ is taken from the set of full measure in which the series (\ref{pi2}) converges and thus the attractor $\Lambda^{\mathbf{y}}$ in Definition \ref{attractor3} is well defined.
\begin{lemma}
Let $0<\epsilon<1-\norm{T}$. We then have for all $\norm{T}+ \epsilon \leq \theta <1$ and for $\mathbb{P}$-almost every $\mathbf{y} \in Y^{\infty}$ there exists an $N \in \mathbb{N}$ such that for $n \geq N$, any finite word $(i_0,\ldots, i_{n-1}) \in J_n$ and any $\mathbf{i} \in [i_0,\ldots, i_{n-1}]$ then $\Pi^{\mathbf{y}}(\mathbf{i}) \in T_{i_0,\ldots,i_{n-1}}(B_{\theta, n})$ where $B_{\theta, n}$ is a ball of radius $\frac{C}{\theta^n}$ where $C$ is independent of $n$ and $\theta$.
\label{ub lemma}
\end{lemma}
\begin{proof}
By Lemma \ref{bc} there exists a set of full measure $X \subset Y^{\infty}$ such that for some $N \in \mathbb{N}$ and all $n \geq N$, and all $\mathbf{y} \in X$, then $|y_{i_0,\ldots,i_{n-1}} | < \frac{1}{\theta^n}$ for all $(i_0,\ldots,i_{n-1})\in\{1,\ldots,m\}^n$. Fix $n \geq N$ and a cylinder $[i_0,\ldots, i_{n-1} ]$. Suppose $\mathbf{i}, \mathbf{j}\in [i_0,\ldots, i_{n-1} ]$ and let $\Pi^{\mathbf{y}}_{\mathbf{i},\mathbf{j}}:= |\Pi^{\mathbf{y}}(\mathbf{i}) - \Pi^{\mathbf{y}}(\mathbf{j})|$. We then get
\begin{eqnarray*}
&\Pi^{\mathbf{y}}_{\mathbf{i},\mathbf{j}}& =\left|\left(a_{i_0} + y_{i_0} + \sum_{k=0}^{\infty}T_{i_0,\ldots,i_k}(a_{i_{k+1}}+y_{i_0,\ldots,i_{k+1}})\right)\right.\\
 &&-\left.\left(a_{j_0} + y_{j_0} + \sum_{k=0}^{\infty}T_{i_0,\ldots,i_k}(a_{i_{k+1}}+y_{j_0,\ldots,j_{k+1}})\right) \right|\\
&=&\left|T_{i_0,\ldots,i_{n-1}} \left(\left(a_{i_n} + y_{i_0,\ldots,i_n}+\sum_{k=n}^{\infty}T_{i_n,\ldots,i_k}(a_{i_{k+1}}+y_{i_0,\ldots,i_{k+1}})\right)\right.\right.\\
&-&\left.\left.\left(a_{j_n} + y_{j_0,\ldots,j_n}+\sum_{k=n}^{\infty}T_{j_n,\ldots,j_k}(a_{j_{k+1}}+y_{j_0,\ldots,j_{k+1}})\right)\right)\right|. 
\end{eqnarray*}
Observe that since $\norm{T}<\theta$
\begin{eqnarray}
&&\left|\left(a_{i_n} + y_{i_0,\ldots,i_n}+\sum_{k=n}^{\infty}T_{i_n,\ldots,i_k}(a_{i_{k+1}}+y_{i_0,\ldots,i_{k+1}})\right)\right. \notag\\ 
&-&\left.\left(a_{j_n} + y_{j_0,\ldots,j_n}+\sum_{k=n}^{\infty}T_{j_n,\ldots,j_k}(a_{j_{k+1}}+y_{j_0,\ldots,j_{k+1}})\right)\right| \notag\\ 
&\leq& 2(\norm{a} + \frac{1}{\theta^{n+1}}) + 2\norm{T}(\norm{a} + \frac{1}{\theta^{n+2}}) + \cdots \notag\\ 
&=& \frac{2\norm{a}}{1- \norm{T}} + \frac{2}{\theta^{n+1}}(1 + \frac{\norm{T}}{\theta} + \cdots)\notag\\ 
&\leq&\frac{2\norm{a}}{1- \norm{T}} + \frac{2}{\theta^{n+1}}\frac{1}{1-\frac{\norm{T}}{\theta}} \leq \frac{2\norm{a}}{1- \norm{T}} + \frac{2}{\theta^{n+1}}\frac{1}{1-\frac{\norm{T}}{\norm{T}+ \epsilon}}\notag\\ 
& \leq& \frac{2 \norm{a}}{1- \norm{T}} + \frac{2}{\theta^n}\frac{1}{\norm{T}(1- \frac{\norm{T}}{\norm{T}+ \epsilon})}. 
\label{r}
\end{eqnarray}

Thus we can fix $\mathbf{i} \in [i_0,...,i_{n-1}]$, and then for any $\mathbf{j} \in [i_0,...,i_{n-1}]$ we have $\Pi^{\mathbf{y}}(\mathbf{j}) \in T_{i_0,...,i_{n-1}}(B)$ where $B= B(\Pi^{\mathbf{y}}(\mathbf{i}), r_{\theta, n})$, where $r_{\theta, n}$ is given by (\ref{r}).  Put 
$$C'= \max\left\{{\frac{2}{\norm{T}(1- \frac{\norm{T}}{\norm{T}+ \epsilon})}, \frac{2\norm{a}}{1- \norm{T}}}\right\}$$ which is clearly independent of $\theta$ and $n$. Then 
$$r_{\theta, n} \leq C'\left(1 + \frac{1}{\theta^n}\right) = C'\left(\frac{\theta^n +1}{\theta^n}\right) \leq 2C'\frac{1}{\theta^n}$$ Putting $C= 2C'$ and setting $B_{\theta, n}= B(\Pi^{\mathbf{y}}(\mathbf{i}), \frac{C}{\theta^n})$ proves the result.
\end{proof}

In Proposition 4.1 in \cite{Falcpaper} it is shown that the limit
$$\lim_{n \to \infty} \left( \sum_{J_n} \phi^s(T_{\mathbf{i}}) \right)^{\frac{1}{n}} $$
is continuous in $s$ and decreasing. This enables us to make the following definition.  
\begin{defi}
Suppose $d(T_1,\ldots,T_m) <d$. Then
$$\lim_{n \to \infty}\left(\sum_{J_n} \phi^d(T_{\mathbf{i}})\right)^{\frac{1}{n}}= t< 1$$ so we can define a monotone increasing sequence $\{\theta_k\}_{k \in \mathbb{N}}$ where\\
 $\lim_{k \to \infty}\theta_k = 1$ such that $\max\{{\norm{T}+ \epsilon, t}\} <\theta_k <1$ for all $k$. Then define $\{s_k\}_{k \in \mathbb{N}}$ to be the monotone decreasing sequence where for each $k$, $d(T_1,\ldots,T_m) < s_k <d$ is defined to be such that $$\lim_{n \to \infty}\left(\sum_{J_n} \phi^{s_k}(T_{\mathbf{i}})\right)^{\frac{1}{n}}= \theta_k^d$$ 
\label{sk}
\end{defi}

\begin{lemma}
The limit of the sequence $\{s_k\}_{k \in \mathbb{N}}$ exists and
$\displaystyle\lim_{k \to \infty}{s_k} = s= d(T_1,\ldots,T_m)$.
\label{sk limit}
\end{lemma}
\begin{proof}
Since $\{s_k\}_{k \in \mathbb{N}}$ is a monotone decreasing sequence that is bounded below by $d(T_1, ..., T_m)$ it converges, so its limit exists. Now, let $\epsilon > 0$ and define $\delta>0$ to satisfy
$$\lim_{n \to \infty}\left(\sum_{J_n}\phi^{s + \epsilon}(T_{\mathbf{i}})\right)^{\frac{1}{n}}= 1- \delta.$$
Choose $N$ such that for $k \geq N$ then $|1 -\theta_k^d|< \delta$ which is possible since $\lim_{k\to\infty}\theta_k=1$ and thus $\lim_{k\to\infty}\theta_k^d=1$. By the definition of $s_k$
$$\lim_{n \to \infty}\left(\sum_{J_n}\phi^{s_k}(T_{\mathbf{i}})\right)^{\frac{1}{n}}= \theta_k^d > 1 - \delta= \lim_{n \to \infty}\left(\sum_{J_n}\phi^{s + \epsilon}(T_{\mathbf{i}})\right)^{\frac{1}{n}}.$$
Since $$\lim_{n \to \infty}\left(\sum_{J_n}\phi^{r}(T_{\mathbf{i}})\right)^{\frac{1}{n}}$$ is decreasing in $r$, it follows that $s_k < s + \epsilon$, that is,  $|s_k - s| < \epsilon$. 
\end{proof}
\begin{lemma}
Consider a self-affine IFS of the form (\ref{ifs}) and the attractor $\Lambda^{\mathbf{y}}$ as in Definition \ref{attractor3}.  Then $\mathbb{P}$-almost surely $\mathrm{dim_H(\Lambda^{\mathbf{y}})} \leq d(T_1, \ldots, T_m)$.
\label{case 1}
\end{lemma}
\begin{proof}
First of all we note that if  $d(T_1,...,T_m) \geq d$ then trivially $\mathrm{dim_H(\Lambda^{\mathbf{y}}) } \leq d$. So we just consider the case where $d(T_1,...,T_m) < d$. Take the sequences $\{\theta_k\}_{k \in \mathbb{N}}$ and $\{s_k\}_{k \in \mathbb{N}}$ as in Definition \ref{sk}. By Lemma \ref{bc} we can find a set $\mathbf{Y}_k \subseteq Y^{\infty}$ of full measure such that $|y_{i_0,\ldots,i_{n-1}}|<\theta_k^n$ for all $(i_0,\ldots,i_{n-1}) \in J_n$ where $n$ is sufficiently large. Put 

$$\mathbf{Y}= \bigcap_{k=1}^{\infty}\mathbf{Y}_k$$ 
which is therefore also a set of full measure. Next fix $k$, and $\mathbf{y} \in \mathbf{Y}$ and let $n$ be sufficiently large so that $y_{i_0,\ldots,i_{n-1}} <\theta_k^n$ for all $\mathbf{i}=(i_0,\ldots,i_{n-1}) \in J_n$. 
We know by Lemma \ref{ub lemma} that for all $\mathbf{j} \in [i_0, \ldots, i_{n-1}]$, $\Pi^{\mathbf{y}}(\mathbf{j})$ is contained in the ball $T_{i_0,\ldots,i_{n-1}}(B_{\theta_k, n})$ which is contained in a parallelepiped with sides of length $\frac{4C}{\theta_k^n}\alpha_1(T_{\mathbf{i}}), \ldots, \frac{4C}{\theta_k^n}\alpha_{d}(T_{\mathbf{i}})$. We let $r=\inf\{z\in\mathbb{Z}:z\geq s_k\}$. 
We can thus divide this parallelepiped into at most $\frac{\alpha_{1}(T_{\mathbf{i}})}{\alpha_{r}(T_{\mathbf{i}})} \cdots\frac{\alpha_{r-1}(T_{\mathbf{i}})}{\alpha_{r}(T_{\mathbf{i}})}$ cubes of side $\alpha_{r}(T_{\mathbf{i}})\frac{4C}{\theta_k^n}$. Thus we take such a collection of cubes for each $\mathbf{i} \in J_n$ as a cover for $\Lambda^{\mathbf{y}}$ where $\mathbf{y} \in \mathbf{Y}_k$. Putting $\delta_n = \max_{\mathbf{i} \in J_n} \{\alpha_{r}(T_{\mathbf{i}})\frac{4C}{\theta_k^n}\}$ we get that
\begin{eqnarray*}
\mathcal{H}^{s_k}_{\delta_n}(\Lambda^{\mathbf{y}}) &\leq& \sum_{\mathbf{i} \in J_n}\frac{(4C)^{s_k}}{\theta_k^{ns_k}}\alpha_{1}(T_{\mathbf{i}})\cdots\alpha_{r-1}(T_{\mathbf{i}})(\alpha_{r}(T_{\mathbf{i}}))^{s_k-r+1}\\
 &=&\sum_{\mathbf{i} \in J_n}\frac{(4C)^{s_k}}{\theta_k^{ns_k}}\phi^{s_k}(T_{\mathbf{i}}) \leq \sum_{\mathbf{i} \in J_n}\frac{(4C)^{s_k}}{\theta_k^{nd}}\phi^{s_k}(T_{\mathbf{i}})
\end{eqnarray*}
since $s_k \leq d$. Letting $n \to \infty$ we get
\begin{displaymath}
\mathcal{H}^{s_k}(\Lambda^{\mathbf{y}}) \leq \lim_{n \to \infty}\frac{(4C)^{s_k}}{\theta_k^{nd}}\sum_{\mathbf{i} \in J_n}\phi^{s_k}(T_{\mathbf{i}}) = \lim_{n \to \infty}\frac{(4C)^{s_k}}{\theta_k^{nd}} \theta_k^{nd}
\end{displaymath}
by definition of $s_k$. So 
\begin{displaymath}
\mathcal{H}^{s_k}(\Lambda^{\mathbf{y}}) \leq (4C)^{s_k} < \infty
\end{displaymath}
and thus $\mathrm{dim_H(\Lambda^{\mathbf{y}})} \leq s_k$ for $\mathbf{y} \in \mathbf{Y}_k$. This means that for any $\mathbf{y}\in \mathbf{Y}$ $\mathrm{dim_H(\Lambda^{\mathbf{y}})} \leq s_k$ and since $s_k$ is a decreasing sequence it follows that $\mathrm{dim_H(\Lambda^{\mathbf{y}})} \leq \inf\{s_k : k \in \mathbb{N}\} = \lim_{n \to \infty} s_k = s$. Thus  $\mathrm{dim_H(\Lambda^{\mathbf{y}})} \leq d(T_1,\ldots,T_m)$ $\mathbb{P}$-almost surely. 
\end{proof}

\section{Proof of the lower bound for Theorem \ref{main}}
The proof of the lower bound for Theorem \ref{main} is fairly similar to the proof of Theorem \ref{JPS} given in \cite{tj}. In particular we use the same method of showing that the self-affine transversality condition holds. However there are some differences in the argument since the perturbations are now distributed according to a measure which may not be compactly supported. In particular, Lemmas \ref{c} is proved in the same way as Lemma 4.5 in \cite{tj} and \ref{lower bound} and \ref{lebesgue} can be deduced from Lemma \ref{trans} in the same way as Proposition 4.4 is proved in \cite{tj}, we give the details here for completeness. One difference is that the projection may not be convergent for all $\mathbf{y}\in Y^{\infty}$.  To overcome this problem we assume random perturbations are in the space $X=X_{\theta}$ defined in the statement of Lemma \ref{bc} where $\theta>\norm{T}$ rather than the whole space $Y^{\infty}$. In this way it follows by Lemma \ref{bc} that this set has full measure and by Lemma \ref{welldefined} that $\Pi^{\mathbf{y}}$ is well defined for all $\mathbf{y}\in X$. This definition also ensures that there are no issues with measurability when we apply Fubini's Theorem. 
\begin{lemma}
There exists a finite measure $\mu$ supported on $J_{\infty}$ and a constant $c'$ such that if $s < d(T_1,\ldots, T_m) $ then
\begin{displaymath}
\mu([\omega]) \leq c' \phi^s(T_{\omega})
\end{displaymath}
for every finite word $\omega \in J$.
\label{existence}
\end{lemma}
\begin{proof}
 See, for example, Lemma 3 in \cite{tj}.
\end{proof}
The following definition was used in \cite{tj} to introduce a self-affine transversality condition.
\begin{defi}
For fixed $\mathbf{i},\mathbf{j}\in J_{\infty}$ define $Z_{\mathbf{i} \wedge \mathbf{j}}:[0,\infty)\to [0,1]$ by
$$Z_{\mathbf{i} \wedge \mathbf{j}}(\rho) := \prod_{k=1}^{d} \frac{min\{\rho, \alpha_k(T_{\mathbf{i} \wedge \mathbf{j}})\}}{\alpha_k(T_{\mathbf{i} \wedge \mathbf{j}})}.$$
\label{z}
\end{defi}

\begin{lemma}
The self-affine transversality condition from \cite{tj} holds for the measure $\mathbb{P}$. That is, there exists $C>0$ such that for all $\mathbf{i},\mathbf{j}\in J_{\infty}$,  
\begin{equation}\label{sat}
\mathbb{P}\{\mathbf{y} \in X : |\Pi^{\mathbf{y}}(\mathbf{i}) - \Pi^{\mathbf{y}}(\mathbf{j})| < \rho\} < C \cdot Z_{\mathbf{i} \wedge \mathbf{j}}(\rho).
\end{equation}
\label{trans}
\end{lemma}
\begin{proof} 
Let $|\mathbf{i} \wedge \mathbf{j}|=n$. We start by noting that
\begin{displaymath}
\mathbb{P}\{\textbf{y} \in X : |\Pi^{\mathbf{y}}(\mathbf{i}) - \Pi^{\mathbf{y}}(\mathbf{j})| < \rho\}= \mathbb{P}\{\textbf{y} \in X : |T_{\mathbf{i} \wedge \mathbf{j}}(y_{i_0,\ldots ,i_n} + q_n(\mathbf{i}, \mathbf{j}, \textbf{y}))| < \rho \}
\end{displaymath}
where 
\begin{eqnarray*} q_n(\mathbf{i}, \mathbf{j}, \textbf{y})&=& a_{i_n} + T_{i_n}(a_{i_{n+1}} + y_{i_0,\ldots, i_{n+1}}) + \cdots\\
& & - (a_{j_n} + y_{i_0,\ldots, j_n} + T_{j_n}(a_{j_{n+1}} + y_{i_0,\ldots, j_{n+1}}) + \cdots).
\end{eqnarray*}
Note that $q_n(\mathbf{i}, \mathbf{j}, \textbf{y})$ is independent of $y_{i_0,\ldots, i_n}$. Thus we can fix all of \textbf{y} except $y_{i_0,\ldots, i_n}$ and for convenience from now on we write $y= y_{i_0,\ldots ,i_n}$, and $q=q_n(\mathbf{i}, \mathbf{j}, \textbf{y})$. Then it is enough to prove the condition for
\begin{displaymath}
\eta\{y \in Y: y \in T_{\mathbf{i} \wedge \mathbf{j}}^{-1}B(T_{\mathbf{i} \wedge \mathbf{j}}q, \rho) \}.
\end{displaymath}
Denote $Box_{\rho} := [a_1 - \rho, a_1 + \rho] \times \cdots \times [a_d - \rho, a_d + \rho]=$

$[\pi_1(T_{\mathbf{i} \wedge \mathbf{j}}(q)) - \rho, \pi_1(T_{\mathbf{i} \wedge \mathbf{j}}(q)) + \rho] \times \cdots \times [\pi_d(T_{\mathbf{i} \wedge \mathbf{j}}(q)) - \rho, \pi_d(T_{\mathbf{i} \wedge \mathbf{j}}(q)) + \rho]$
where $\pi_k$ are the projections to the $x_k$ axes. Then
\begin{displaymath}
\eta\{y \in X: y \in T_{\mathbf{i} \wedge \mathbf{j}}^{-1}B(T_{\mathbf{i} \wedge \mathbf{j}}q, \rho) \} \leq \eta\{y \in T_{\mathbf{i} \wedge \mathbf{j}}^{-1}Box_{\rho}\}.
\end{displaymath}
Let $\{x_{\varphi(1)}, \ldots, x_{\varphi(d)}\}$ be the orthonormal elements in the new basis given by the rotation in $T_{\mathbf{i} \wedge \mathbf{j}}^{-1}$, such that the axes $x_{\varphi(k)}$ correspond to $\alpha_k$ in the following way: the principal semi axes of $T_{\mathbf{i} \wedge \mathbf{j}}(B) $ of length $\alpha_k$ lies along the axis $x_{\varphi(k)}$. Let $\pi_{\theta(k)}$ denote the projection to the $k$-dimensional plane that lies along the $\{x_{\varphi(1)},\ldots, x_{\varphi(k)}\}$ axes. Denote $\alpha_0= \infty$ and $\alpha_{d+1}=0$. Then for $\alpha_{k+1} \leq \rho \leq \alpha_{k}$, where $0 \leq k \leq d$, estimate 
$$\eta\{y \in T_{\mathbf{i} \wedge \mathbf{j}}^{-1}Box_{\rho}\} \leq \eta_{\theta(k)}\{\pi_{\theta(k)}T_{\mathbf{i} \wedge \mathbf{j}}^{-1}Box_{\rho}\}$$
where $\eta_{\theta(k)}$ is the pushforward measure $\pi_{\theta(k)\ast} \eta$.
Since $\eta_{\theta(k)}$ is by assumption absolutely continuous with respect to $\lambda_k$ (where this time $\lambda_k$ is the Lebesgue measure on a $k$-dimensional subspace of $\mathbb{R}^d$) and has bounded density with respect to $\lambda_k$ we have
\begin{eqnarray*}
\eta_{\theta(k)}\{\pi_{\theta(k)}T_{\mathbf{i} \wedge \mathbf{j}}^{-1}Box_{\rho}\} &\leq& K \lambda_k \{\pi_{\theta(k)}T_{\mathbf{i} \wedge \mathbf{j}}^{-1}Box_{\rho}\}\\
 &\leq& K \lambda_k \{\pi_{\theta(k)}T_{\mathbf{i} \wedge \mathbf{j}}^{-1}[-\rho, \rho]^d\} \\
&\leq& 2^k K \frac{\rho^k}{\alpha_1\ldots\alpha_k}
\end{eqnarray*}
where $K$ is a constant defined directly below Definition \ref{condition}. Putting $C=2^d K$ we get that $C$ is a constant independent of $\mathbf{i}$ and $\mathbf{j}$ and
\begin{displaymath}
\mathbb{P}\{\mathbf{y} \in X : | \Pi^{\mathbf{y}}(\mathbf{i}) -\Pi^{\mathbf{y}}(\mathbf{j})| < \rho \} < C \cdot Z_{\mathbf{i} \wedge \mathbf{j}}(\rho).
\end{displaymath} 
\end{proof}
We use the self-affine transversality condition in order to derive the following inequality, towards finding a lower bound for the Hausdorff dimension. 

\begin{lemma}
We have that for all non-integral $t\in (0,d)$ and for any $ \mathbf{i},\mathbf{j}\in J_{\infty}$,
\begin{displaymath}
\int_{\textbf{y} \in X} |\Pi^{\textbf{y}}(\mathbf{i})-\Pi^{\textbf{y}}(\mathbf{j})|^{-t} \mathrm{d}\mathbb{P} < \frac{c}{\phi^t(T_{\mathbf{i} \wedge \mathbf{j}})}.
\end{displaymath}
\label{c}
\end{lemma}

\begin{proof}
 We can write 
$$
\int_{\textbf{y} \in X} |\Pi^{\textbf{y}}(\mathbf{i})-\Pi^{\textbf{y}}(\mathbf{j})|^{-t} \mathrm{d}\mathbb{P}  = t\int_{\rho=0}^{\infty}\mathbb{P}\{\textbf{y} \in X : |\Pi^{\mathbf{y}}(\mathbf{i})- \Pi^{\mathbf{y}}(\mathbf{j})| < \rho \}\rho^{-t-1}\mathrm{d}\rho.
$$
Thus, by the self-affine transversality condition, (\ref{sat}), it is enough to show that 
\begin{displaymath}
\int_{\rho=0}^{\infty} Z_{\mathbf{i} \wedge \mathbf{j}}(\rho)\rho^{-t-1}\mathrm{d}\rho \leq \frac{c}{\phi^t(T_{\mathbf{i} \wedge \mathbf{j}})}.
\end{displaymath}
Let $k$ be such that $k-1 < t < k$ and write 
\begin{displaymath}
\int_{\rho=0}^{\infty} Z_{\mathbf{i} \wedge \mathbf{j}}(\rho)\rho^{-t-1}d\rho = \int_{\rho=0}^{\alpha_k} Z_{\mathbf{i} \wedge \mathbf{j}}(\rho)\rho^{-t-1}\mathrm{d}\rho  + \int_{\rho= \alpha_k}^{\infty} Z_{\mathbf{i} \wedge \mathbf{j}}(\rho)\rho^{-t-1}\mathrm{d}\rho.
\end{displaymath}
We begin by dealing with the first integral. For each $i \geq k+1$ note that $\rho < \alpha_{i}$ implies that $\frac{\rho^i}{\alpha_1\ldots \alpha_i}= \frac{\rho^{i-1}\frac{\rho}{\alpha_i}}{\alpha_1\ldots \alpha_{i-1}\frac{\alpha_i}{\alpha_i}} < \frac{\rho^{i-1}}{\alpha_1\ldots \alpha_{i-1}}$ since $\rho < \alpha_i \Leftrightarrow \frac{\rho}{\alpha_i} < 1$. Since $\alpha_{i+1} <\rho < \alpha_{i}$ implies that $\rho < \alpha_{j}$ for all $k+1 \leq j \leq i$ it follows that 
\begin{displaymath}
Z_{\mathbf{i} \wedge \mathbf{j}}(\rho) < \frac{\rho^k}{\alpha_1\ldots\alpha_k}
\end{displaymath}
for all $\rho < \alpha_k$. Inserting this into the first integral we get 
\begin{eqnarray*}
\int_{\rho=0}^{\alpha_k} Z_{\mathbf{i} \wedge \mathbf{j}}(\rho)\rho^{-t-1}\mathrm{d}\rho  &<& \int_{\rho=0}^{\alpha_k} \frac{\rho^k}{\alpha_1\ldots\alpha_k}\rho^{-t-1}\mathrm{d}\rho =\frac{1}{k-t} [ \alpha_1\ldots\alpha_k^{t-(k-1)}]^{-1}\\
&=&\frac{1}{k-t}\frac{1}{\phi^t(T_{\mathbf{i} \wedge \mathbf{j}})}.
\end{eqnarray*}
Next, we move onto finding an upper bound for the second integral,
\begin{eqnarray*}
\int_{\rho= \alpha_k}^{\infty} Z_{\mathbf{i} \wedge \mathbf{j}}(\rho)\rho^{-t-1}\mathrm{d}\rho&=& \int_{\rho= \alpha_k}^{\alpha_{k-1}} Z_{\mathbf{i} \wedge \mathbf{j}}(\rho)\rho^{-t-1}\mathrm{d}\rho + \cdots + \int_{\rho= \alpha_1}^{\infty} Z_{\mathbf{i} \wedge \mathbf{j}}(\rho)\rho^{-t-1}\mathrm{d}\rho\\
&=&\sum_{l=0}^{k-1} \int_{\rho= \alpha_{l+1}}^{\alpha_{l}} Z_{\mathbf{i} \wedge \mathbf{j}}(\rho)\rho^{-t-1}\mathrm{d}\rho
\end{eqnarray*}
where as usual we take $\alpha_0= \infty$. Integrating we get
\begin{align*}
&&\frac{1}{t}(\alpha_1^{-t} - \alpha_0^{-t})\\
&+&\frac{1}{t-1}\frac{1}{\alpha_1}(\alpha_2^{1-t} - \alpha_1^{1-t}) \\
&+&\frac{1}{t-2}\frac{1}{\alpha_1\alpha_2}(\alpha_3^{2-t} - \alpha_2^{2-t}) \\
&\vdots&\\
&+&\frac{1}{t+1-k}\frac{1}{\alpha_1\alpha_2\ldots\alpha_{k-1}}(\alpha_k^{k-1-t} - \alpha_{k-1}^{k-1-t} )
\end{align*}
Noting that the second term in each line will always be greater than the first term in the previous line, we use a diagonal argument to deduce that
\begin{eqnarray*}
\int_{\rho= \alpha_k}^{\infty} Z_{\mathbf{i} \wedge \mathbf{j}}(\rho)\rho^{-t-1}\mathrm{d}\rho &<& \frac{1}{t+1-k}\frac{1}{\alpha_1\alpha_2\ldots\alpha_{k-1}}(\alpha_k^{k-1-t} )\\
&=& \frac{1}{t+1-k}(\alpha_1\alpha_2\ldots\alpha_{k-1}\alpha_k^{t-k+1} )^{-1}\\
 &=&\frac{1}{t+1-k}\frac{1}{\phi^t(T_{\mathbf{i} \wedge \mathbf{j}})}
\end{eqnarray*}
Thus we get 
\begin{displaymath}
\int_{\rho=0}^{\infty} Z_{\mathbf{i} \wedge \mathbf{j}}(\rho)\rho^{-t-1}\mathrm{d}\rho < \frac{1}{(k-t)(t+1-k)}\frac{1}{{\phi^t(T_{\mathbf{i} \wedge \mathbf{j}})}}
\end{displaymath}
where clearly $\frac{1}{(k-t)(t+1-k)}$ is independent of $\mathbf{i}$ and $\mathbf{j}$, which proves the result. 
\end{proof}

\begin{lemma}
Consider the self-affine IFS of the form (\ref{ifs}) and attractor $\Lambda^{\mathbf{y}}$ as in Definition \ref{attractor3}. Then $\mathbb{P}$-almost surely $\mathrm{dim_H(\Lambda^{\mathbf{y}})} \geq d(T_1, \ldots, T_m)$.
\label{lower bound}
\end{lemma}
\begin{proof}
 We use the potential theoretic characterisation of Hausdorff dimension. Let $t<s<d(T_1,\ldots,T_m)$ be chosen such that $t\notin\Z$. We need to show that there exists a finite measure $\mu$ supported on $J_{\infty}$ such that for $\mathbb{P}$-almost all $\textbf{y} \in X$,
\begin{displaymath}
\int\int_{(\mathbf{i}, \mathbf{j}) \in J_{\infty} \times J_{\infty}}|\Pi^{\mathbf{y}}(\mathbf{i}) - \Pi^{\mathbf{y}}(\mathbf{j})|^{-t}\mathrm{d}\mu (\mathbf{i}) \mathrm{d}\mu(\mathbf{j}) < \infty
\end{displaymath}
Equivalently, we need to show that the triple integral
\begin{displaymath}
\int_{\mathbf{y} \in X}\int\int_{(\mathbf{i}, \mathbf{j}) \in J_{\infty} \times J_{\infty}} |\Pi^{\mathbf{y}}(\mathbf{i}) - \Pi^{\mathbf{y}}(\mathbf{j})|^{-t}\mathrm{d}\mu(\mathbf{i})\mathrm{d}\mu(\mathbf{j})\mathrm{d}\mathbb{P}(\mathbf{y}) < \infty
\end{displaymath}
Take $\mu$ that satisfies Lemma \ref{existence}. By Fubini's theorem and Lemma \ref{c} it suffices to show that the following is finite:
\begin{eqnarray*}
&&\int\int_{(\mathbf{i}, \mathbf{j}) \in J_{\infty} \times J_{\infty}}\int_{\mathbf{y} \in X} |\Pi^{\mathbf{y}}(\mathbf{i}) - \Pi^{\mathbf{y}}(\mathbf{j})|^{-t}\mathrm{d}\mathbb{P}(\mathbf{y})\mathrm{d}\mu(\mathbf{i})d\mu(\mathbf{j})\\
 &<& c\int\int_{(\mathbf{i}, \mathbf{j}) \in J_{\infty} \times J_{\infty}} \frac{1}{\phi^t(T_{\mathbf{i} \wedge \mathbf{j}})}\mathrm{d}\mu(\mathbf{i})\mathrm{d}\mu(\mathbf{j}).
\end{eqnarray*}
Rewriting $\mathbf{i} \wedge \mathbf{j}= \omega$ and recalling that the measure $\mu$ was chosen to satisfy Lemma \ref{existence}, we get that the above is equal to
\begin{displaymath}
c\sum_{k=0}^{\infty} \sum_{|\omega|=k}\int\int_{\mathbf{i} \wedge \mathbf{j} = \omega} \frac{1}{\phi^t(T_{\omega})}\mathrm{d}\mu(\mathbf{i})\mathrm{d}\mu(\mathbf{j}) \leq c\sum_{k=0}^{\infty} \sum_{|\omega|=k}\mu([\omega])^2\frac{1}{\phi^t(T_{\omega})}
\end{displaymath}
\begin{equation}
\leq cc'\sum_{k=0}^{\infty} \sum_{|\omega|=k}\mu([\omega])\phi^s(T_{\omega})\frac{1}{\phi^t(T_{\omega})}
\label{lb}
\end{equation}
Now, we use some properties of the singular value function to bound this above. Choose $a$, $b$ such that for all $i \in \{1,\ldots, m\}$ we have 
$$0<a \leq\alpha_d(T_i) \leq \cdots \leq \alpha_1(T_i) \leq b<1.$$ 
In particular, since the singular value function is submultiplicative it follows that for all finite words $\textbf{i} \in J_n$ we have 
\begin{equation}
\phi^s(T_{\textbf{i}})= \phi^s(T_{i_1}\circ\cdots\circ T_{i_n}) \leq \phi^s(T_{i_1})\cdots\phi^s(T_{i_n}) \leq b^n= b^{|\textbf{i}|}
\label{submult}
\end{equation}
Secondly, note that since for any $s, h >0$ and matrix $T$ we have $\phi^{s+h}(T) \leq \phi^s(A) \alpha_1^h(T)$, then dividing this through by $\phi^s(T)$, fixing $h= s - t >0$ and $T= T_{\omega}$ we get that 
\begin{equation}
\frac{\phi^s(T_{\omega})}{\phi^t(T_{\omega})}=\frac{\phi^{t +(s-t)}(T_{\omega})}{\phi^t(T_{\omega})} \leq \alpha_1^{s-t}(T_{\omega}) \leq b^{|\omega|(s-t)}
\label{b}
\end{equation}
by (\ref{submult}). We put this into (\ref{lb}) to get
\begin{eqnarray*}
&&cc'\sum_{k=0}^{\infty} \sum_{|\omega|=k}\mu([\omega])\phi^s(T_{\omega})\frac{1}{\phi^t(T_{\omega})}\\
 &&\leq cc'\sum_{k=0}^{\infty} \sum_{|\omega|=k}\mu([\omega]) b^{k(s-t)}= cc'\sum_{k=0}^{\infty} b^{k(s-t)} < \infty
\end{eqnarray*}
since $b^{s-t} < 1$ whenever $s>t$ thus this geometric progression converges. Since we can put $t$ arbitrarily close to $s$ which in turn we can place arbitrarily close to $d(T_1,\ldots,T_m)$, by the potential theoretic characterisation of Hausdorff dimension it follows that $\textrm{dim}_H(\Lambda^{\mathbf{y}}) \geq t$ for any non-integer $t<d(T_1,\ldots,T_m)$, in other words, $\textrm{dim}_H(\Lambda^{\mathbf{y}}) \geq d(T_1,\ldots,T_m)$.
 \end{proof}

\begin{lemma}
Consider the self-affine IFS of the form (\ref{ifs}) and the attractor $\Lambda^{\mathbf{y}}$ as in Definition \ref{attractor3}. If $d(T_1,\ldots,T_m) > d$ then for $\mathbb{P}$-almost all $\mathbf{y}$ then $\lambda_d(\Lambda^{\mathbf{y}}) >0$
\label{lebesgue}
\end{lemma}
\begin{proof}
 Let $\Pi_{\star}^{\mathbf{y}}\mu$ be the natural projection of the measure $\mu$ defined in \ref{existence}. It is clearly enough to show that $\Pi_{\star}^{\mathbf{y}}\mu$ is absolutely continuous with respect to $\lambda_d$ for $\mathbb{P}$-almost all $\mathbf{y}$. We follow a standard approach (introduced by Peres and Solomyak in \cite{ps}) to show absolute continuity of $\Pi_{\star}^{\mathbf{y}}\mu$ for $\mathbb{P}$-almost all $\mathbf{y}$. In particular it suffices to show that 
$$I:= \int_X \int \liminf_{r \to 0} \frac{\Pi_{\star}^{\mathbf{y}}\mu(B(x, r))}{r^d}\mathrm{d}\Pi_{\star}^{\mathbf{y}}\mu \mathrm{d}\mathbb{P}(\mathbf{y}) < \infty.$$
By Fatou's lemma
\begin{eqnarray*}
I &\leq& \liminf_{r \to 0} \frac{1}{r^d} \int_X \int \Pi_{\star}^{\mathbf{y}}\mu(B(x, r)) \mathrm{d}\Pi_{\star}^{\mathbf{y}}\mu \mathrm{d}\mathbb{P}(\mathbf{y}) \\
&=& \liminf_{r \to 0} \frac{1}{r^d} \int_X \int \int \chi_{(x' \in B(x, r))} \mathrm{d}\Pi_{\star}^{\mathbf{y}}(x')\mu \mathrm{d}\Pi_{\star}^{\mathbf{y}}(x)\mu \mathrm{d}\mathbb{P}(\mathbf{y})\\
&=& \liminf_{r \to 0} \frac{1}{r^d} \int_X \int \int \chi_{\{(\mathbf{i},\mathbf{j}) : |\Pi^{\mathbf{y}}(\mathbf{i}) - \Pi^{\mathbf{y}}(\mathbf{j})| < r\} } \mathrm{d}\mu(\mathbf{i}) \mathrm{d}\mu(\mathbf{j}) \mathrm{d}\mathbb{P}(\mathbf{y})\\
&\leq& \liminf_{r \to 0} \frac{1}{r^d} \int \int \mathbb{P}\{\textbf{y} \in X : |\Pi^{\mathbf{y}}(\mathbf{i}) - \Pi^{\mathbf{y}}(\mathbf{j})| < r\}\mathrm{d}\mu(\mathbf{i}) \mathrm{d}\mu(\mathbf{j}) 
\end{eqnarray*}
by Fubini's theorem. Next note that by definition
$$Z_{\mathbf{i} \wedge \mathbf{j}}(\rho) := \prod_{k=1}^{d} \frac{\min\{\rho, \alpha_k(T_{\mathbf{i} \wedge \mathbf{j}})\}}{\alpha_k(T_{\mathbf{i} \wedge \mathbf{j}})} \leq \frac{\rho^d}{\alpha_1(T_{\mathbf{i} \wedge \mathbf{j}})\cdots\alpha_d(T_{\mathbf{i} \wedge \mathbf{j}})}$$
so that by Lemma \ref{trans}, 
$$\mathbb{P}\{\textbf{y} \in X : |\Pi^{\mathbf{y}}(\mathbf{i}) - \Pi^{\mathbf{y}}(\mathbf{j})| < r\} < C \cdot \frac{r^d}{\alpha_1(T_{\mathbf{i} \wedge \mathbf{j}})\cdots\alpha_d(T_{\mathbf{i} \wedge \mathbf{j}})}.$$ Thus
\begin{eqnarray*}
I &\leq& C \liminf_{r \to 0} \frac{1}{r^d} \int \int \frac{r^d}{\alpha_1(T_{\mathbf{i} \wedge \mathbf{j}})\cdots\alpha_d(T_{\mathbf{i} \wedge \mathbf{j}})} \mathrm{d}\mu(\mathbf{i}) \mathrm{d}\mu(\mathbf{j})\\
&\leq& C \sum_{k=0}^{\infty} \sum_{|\omega|=k} \frac{\mu([\omega])^2}{\phi^d(T_{\omega})}\\
&\leq&c'C\sum_{k=0}^{\infty} \sum_{|\omega|=k} \frac{\phi^s(T_{\omega})}{\phi^d(T_{\omega})}\mu([\omega])\\
&\leq&c'C\sum_{k=0}^{\infty} \sum_{|\omega|=k} \frac{\phi^s(T_{\omega})}{\phi^d(T_{\omega})}\mu([\omega])\\
&\leq&c'C\sum_{k=0}^{\infty} b^{k(s-d)} \sum_{|\omega|=k} \mu([\omega])
\end{eqnarray*}
by the choice of $\mu$ in Lemma \ref{existence} and by (\ref{b}). Thus
\begin{eqnarray*}
I &<& \infty
\end{eqnarray*}
since $s>d$.
This proves the result. \end{proof}
\noindent \textit{Proof of Theorem \ref{main}}: This is a direct consequence of Lemmas \ref{case 1}, \ref{lower bound} and \ref{lebesgue}. \qed
\section{Comments and questions}
We conclude this note by making a few comments about possible extensions of this work.
\begin{enumerate}
\item[1.]
In Theorem 3 in \cite{tj} a result on the dimension and absolute continuity of the projection of an ergodic measure from the shift space is considered. The same result should hold where the random perturbations are distributed independently according to a measure $\eta$ satisfying Definition \ref{condition}.
\item[2.]
The conditions on $\eta$ in Definition \ref{condition} can be relaxed if a different model of randomness is used. If rather than allowing each term $y_{i_0,\ldots,i_{n-1}}$ to be distributed according to $\eta$ we assume that $y_{i_0,\ldots,i_{n-1}}=y_{j_0,\ldots,j_{n-1}}$ whenever $i_{n-1}=j_{n-1}$ we can relax the conditions on $\eta$. In particular this means that the condition on $\eta$ required is that for all $\theta<1$ we have $\sum_{n=1}^{\infty}m\eta\{y\in\R^d:|y|\geq\theta^{-n}\}<\infty$ and thus we can consider distributions with a much weaker condition on the rate of the decay of the tails than in Definition \ref{condition}. This comes from the fact that with this new model we'll have that if $A_n$ is defined as in the proof of Lemma \ref{bc} then 
$$\mathbb{P}(A_n) \leq m \eta \{y \in \mathbb{R}^d : |y| \geq \theta^{-n} \}$$
since there are only $m$ distinct perturbations on the $n$-th level, and, for the Borel-Cantelli argument to work we need the infinite sum of probabilities $\sum_{n=1}^{\infty} \mathbb{P}(A_n)$ to be finite.
\item[3.]
In \cite{Falcpaper3} generalised dimensions of self-affine measures with random perturbations are considered. The random perturbations in this paper are defined in the same way as in \cite{tj}. It should be possible to extend the results in \cite{Falcpaper3} to non-compactly supported permutations satisfying the assumptions in Definition \ref{condition}.

\end{enumerate}

\end{document}